\documentclass[a4paper,1	2pt,reqno]{amsart}
\usepackage{amsmath,amssymb,mathtools} 
\usepackage[british]{babel}
\usepackage[final,nopatch=footnote]{microtype}
\usepackage{mlmodern}
\usepackage[all,cmtip]{xy}
\usepackage[margin=1in]{geometry}
\usepackage[UKenglish]{isodate}
\usepackage[shortlabels]{enumitem}
\usepackage[hidelinks]{hyperref}
\usepackage{epigraph}
\usepackage{footnote}
\usepackage[T1]{fontenc}
\usepackage{bm}

\swapnumbers
\date{10th May 2026}
\subjclass[2020]{18N30}
\makeatletter
\let\c@subsection\c@equation
\makeatother
\numberwithin{equation}{section}

\theoremstyle{plain}
\newtheorem{theorem}[subsection]{Theorem}
\newtheorem{proposition}[subsection]{Proposition}
\newtheorem{lemma}[subsection]{Lemma}
\newtheorem{corollary}[subsection]{Corollary}

\theoremstyle{definition} 
\newtheorem{definition}[subsection]{Definition}
\newtheorem{example}[subsection]{Example}

\theoremstyle{remark}
\newtheorem{remark}[subsection]{Remark}
\newtheorem{observation}[subsection]{Observation}

\newcommand{\cd}[2][]{\vcenter{\hbox{\xymatrix#1{#2}}}}

\newcommand{\hdash}{\rotatebox[origin=c]{90}{$\vdash$}}
\newcommand{\fatpullbackcorner}[1][dr]{\save*!/#1-1.5pc/#1:(-1,1)@^{|-}\restore}

\title[Parity complexes redux]{Parity complexes redux}
\author{Alexander Campbell}
\address{School of Mathematics and Statistics F07 \\ University of Sydney \\ NSW 2006 \\ Australia}
\email{campbell@maths.usyd.edu.au}
\urladdr{https://acmbl.github.io/}
\dedicatory{Rosso Street octogenario magna cum admiratione}

\begin{document}

\begin{abstract}
We fix the notion of parity complex by a judicious selection from among the axioms originally considered by Street. We show that parity complexes so defined, 
together with the morphisms of parity complexes defined by Verity, form a category equivalent to the category of strong Steiner complexes (\textit{nés} augmented directed complexes with strongly loop-free unital bases). 
To this end, we isolate the purely combinatorial structure possessed by the bases of free augmented directed complexes. This analysis reveals the essential advantage of Steiner's formalism to be that the role of subsets in Street's formalism is played instead by multisets.
\end{abstract}
\maketitle

\epigraph{\flushright\textit{amissam classem, socios a morte reduxi}}{--- Vergil, \textit{Aeneid} 4.375}

\tableofcontents

\section{Introduction}

 A parity complex \cite{paritycomplexes} is a kind of higher-dimensional loop-free directed graph that freely generates a combinatorially defined $\omega$-category. Parity complexes arose within a research programme in higher category theory initiated by Roberts's insight \cite[\S 2]{roberts} that 
 $n$-categories  are the natural coefficient objects for non-abelian cohomology of degree $n$.\footnote{A similar insight later became the starting point of Grothendieck's pursuit of stacks \cite{pursuingstacks}.}  
Such a notion of non-abelian cohomology was achieved by Street in \cite[\S 6]{paritycomplexes},\footnote{Fuller details, dependent on properties of the Gray tensor product of $\omega$-categories, were provided in Street's 1995 Oberwolfach lectures; cf.\ \cite[\S 9]{catcomb}.} where he defined the cohomology $\omega$-category $\mathcal{H}(X,A)$ of a simplicial object $X$ in a category $\mathcal{E}$ with coefficients in an $\omega$-category $A$ in $\mathcal{E}$, which may be expressed as the end $\omega$-category
$$\mathcal{H}(X,A) = \int_{[n]\in\Delta}\operatorname{Fun}^\mathrm{lax}(\mathcal{O}^n, \mathcal{E}(X_n,A));$$
here $\mathcal{O}^n$ denotes the $n$\textsuperscript{th} oriental, i.e.\ the free $n$-category on the parity $n$-simplex, and $\operatorname{Fun}^\mathrm{lax}$ denotes the left hom for the Gray tensor product of $\omega$-categories.\footnote{In the case $\mathcal{E} = \mathbf{Set}$, this expression simplifies to
$\mathcal{H}(X,A) = \operatorname{Fun}^\mathrm{lax}(\mathcal{O}[X],A)$, where $\mathcal{O}$  denotes the left adjoint to Street's simplicial nerve functor for $\omega$-categories \cite[\S 5]{streetaos}; cf.\ \cite[\S 1.4]{mythesis}.}

The orientals are a family of combinatorially defined $\omega$-categories constructed by Street in \cite{streetaos} for the purpose of defining a simplicial nerve functor for $\omega$-categories. The Gray tensor product of $\omega$-categories is part of a biclosed monoidal structure on the category of $\omega$-categories (first studied by Al-Agl and Steiner in \cite{asnerves}) which ultimately depends for its construction on another family of combinatorially defined $\omega$-categories (see \cite[\S 9]{catcomb}), namely the free $n$-categories on the parity $n$-cubes (themselves first constructed by Aitchison and Street, cf.\ \cite{aitchison, paritycomplexes, streetmyhill}). 

Street introduced the theory of parity complexes as a general framework for managing the combinatorics of such families of $\omega$-categories, abstracting the previous treatments  of the orientals and cubes. He was especially interested in the Gray tensor products of orientals with globes (where the $n$-globe is the free $n$-category containing an $n$-cell), which featured in his original definition of  cohomology $\omega$-categories. For this reason he wanted parity complexes to be closed under a suitable notion of product; in his own words \cite[p.\ 316]{paritycomplexes}: ``The abstraction arose purely from the desire to deal with products of examples already understood.''\footnote{For more of Street's own expositions of the history of these developments, see the introductions to \cite{streetaos,paritycomplexes,catcomb} and the conspectus \cite{conspectus}.} 

And yet the notion of parity complex defined in \cite[\S 1]{paritycomplexes} is not closed under products. For Street chose in that paper to defer to the penultimate section \cite[\S 5]{paritycomplexes} the treatment of products of parity complexes and the introduction of the strong loop-freeness axiom -- antisymmetry of the ``solid triangle order'' $\blacktriangleleft$ -- that ensures that parity complexes are closed under products. It is nevertheless evident, both from the Introduction of \cite{paritycomplexes} and from  remarks made in \cite[pp.\ 539, 561]{catcomb}, that Street intended parity complexes to be closed under products and strong loop-freeness to be one of the axioms for parity complexes.\footnote{Street has on occasion (e.g.\ \cite[p.\ 252]{conspectus}), following a suggestion of Eilenberg, taken \emph{linearity} of the solid triangle order (a property enjoyed by all the important examples) as an axiom for parity complexes, but it is theoretically more convenient to ask only for antisymmetry, as, e.g., this leaves the class of parity complexes closed under taking skeleta.}

This expository choice was rendered problematic by the subsequent discovery of counterexamples to certain statements in  \cite[\S\S3--4]{paritycomplexes}, revealing inadequacies in the initial set of axioms given in \cite[\S 1]{paritycomplexes}. First, Verity found a simple counterexample to \cite[Lemma 3.2(a)]{paritycomplexes},\footnote{For the record, a counterexample occurs in the parity complex presenting the 5th oriental, taking the cell $(M,P)$ to be the $1$-target of the cell $\left<012345\right>$ and $X$ to be the set consisting of $(012)$ and $(345)$.} which fault Street was able to repair in \cite{pccorr} by the introduction of a new ``tightness'' hypothesis, which he showed to follow from the globularity and strong loop-freeness axioms. 
A subtler and more fatal counterexample was later discovered by Forest, who produced a parity complex (see \cite[pp.\ 13f.\@]{forest}) which satisfies both the  axioms of \cite[\S1]{paritycomplexes} and the tightness hypothesis of \cite{pccorr}, but for which the ``freeness'' conclusion of Street's main theorem on parity complexes \cite[Theorem 4.2]{paritycomplexes} does not hold. Crucially, Forest's parity complex does not satisfy the strong loop-freeness axiom; as we shall see in the final section of this paper, Street's main theorem does hold for parity complexes satisfying the strong loop-freeness axiom, from which we may conclude  that Street's intended notion of parity complex is fit for purpose.    

Our goal in this paper is to clear these weeds from the theory of parity complexes by fixing the notion of parity complex and placing it in its modern mathematical context, which for us means Steiner's theory of augmented directed complexes \cite{steineromegacats}. For more than a decade now Steiner's theory has played the role in higher category theory  originally envisioned for parity complexes (see e.g.\ \cite{arajoint} and \cite{orientedcategorytheory}), and rightly so:\ we regard Steiner's theory as an amelioration of the theory of parity complexes; indeed, Steiner himself states that a central construction of his theory (the right adjoint functor $\nu$ from augmented directed complexes to $\omega$-categories) ``is modelled on Street's theory of parity complexes'' \cite[p. 428]{steinerorientals}.

We seek to make precise this relationship between parity complexes and augmented directed complexes. To this end, following a recollection  in \S\ref{prelim} of the basics of Steiner theory (as it is affectionately called nowadays), in \S\ref{secapc} we recast the structure of free augmented directed complexes into a purely combinatorial notion which we call an ``additive parity complex'', which differs from the structure of a parity complex in that its elements have multisets, rather than subsets, of negative and positive faces. 
In \S\ref{secwpc} we introduce ``weak parity complexes'', and show that they, together with the morphisms of parity complexes defined by Verity \cite{complicial}, form a category equivalent to a full subcategory of the category of augmented directed complexes. In \S\ref{secpc} we give our axioms for parity complexes proper, which include Street's stronger loop-freeness axiom, and show that these correspond under the above equivalence of categories to the strong Steiner complexes, i.e.\ to the augmented directed complexes with strongly loop-free unital bases.

\section{Preliminaries on Steiner theory} \label{prelim}
We preface our study of parity complexes with a somewhat idiosyncratic recollection of the basics of Steiner's theory of augmented directed complexes \cite{steineromegacats}. We begin with some notions common to both Steiner theory and directed algebraic topology (cf.\ \cite[\S 2.1.1]{grandis}) and conclude by observing (in Propositions \ref{triangleorderprop} and \ref{triangleorderprop2}) that a pair of results of Steiner theory due to Forest \cite{forest} hold under a weaker hypothesis than previously stated in the literature.

\subsection{Preordered abelian groups} \label{sectpag}

A \emph{preordered abelian group} is an abelian group $A$ equipped with a preorder relation $\leq$ compatible with addition, in the sense that $a \leq b$ implies $a+c \leq b+c$ for all $a,b,c\in A$. Such a preorder relation on $A$ is uniquely determined by its \emph{positive cone} $A^* \subseteq A$, the submonoid of positive elements ($a\in A$ such that $a \geq 0$), since $a \leq b$ if and only if $b-a\in A^*$; moreover, every submonoid of an abelian group arises in this way. A group homomorphism between preordered abelian groups is \emph{monotone} if it respects the preorder relations, or equivalently if it preserves positive elements; these are the morphisms in the category of preordered abelian groups. 

A \emph{basis} of a preordered abelian group $A$ is a subset $B \subseteq A^*$ which is a basis for both the commutative monoid $A^*$ and the abelian group $A$. Note that a basis of a preordered abelian group is unique if it exists, since the same is true of commutative monoids; the basis elements may be characterised intrinsically as the minimal strictly positive elements. We call a preordered abelian group \emph{free} if it admits a basis. The functor that sends a preordered abelian group to its set of positive elements admits a left adjoint which sends a set $B$ to a free preordered abelian group with basis isomorphic to $B$.

Let $A$ be a free preordered abelian group with basis $B$. Then $A$ is canonically isomorphic to the free abelian group $\mathbb{Z}B$ of functions $B \longrightarrow \mathbb{Z}$ of finite support under pointwise addition, equipped with the pointwise natural order. It follows that $A$ is moreover a lattice-ordered abelian group, whose join $\vee$ and meet $\wedge$ operations correspond under the canonical isomorphism to the pointwise operations of $\max$ and $\min$ in $\mathbb{Z}B$. In particular, each element $x \in A$ can be expressed uniquely as the difference of two disjoint positive elements, i.e.\ as $x = x_+ - x_-$, with $x_-,x_+ \in A^*$ and $x_- \wedge x_+ = 0$; here $x_+ = x\vee 0$ and $x_- = (-x)\vee 0$; we call $x_-$ and $x_+$ the \emph{negative} and \emph{positive parts} of $x$ respectively. 

\subsection{Directed chain complexes} \label{dccs}
A \emph{directed chain complex} is a positively-graded chain complex of abelian groups $K$ for which each chain group $K_n$ ($n \geq 0$) is a preordered abelian group. Note that the boundary maps of a directed chain complex are not assumed to be monotone. A morphism of directed chain complexes $K \longrightarrow L$ is defined to be a chain map such that each homomorphism $f_n \colon K_n \longrightarrow L_n$ is monotone. This defines the category $\mathbf{dCh}$ of directed chain complexes. 

A \emph{basis} of a directed chain complex $K$ is a graded subset $B$ of $K$, i.e.\ a family of subsets $B_n \subseteq K_n$ for $n \geq 0$, such that each subset $B_n$ is a basis of the preordered abelian group $K_n$. It follows that a basis of a directed chain complex is unique if it exists. We call a directed chain complex \emph{free} if it admits a basis. 

Let $K$ be a free directed chain complex with basis $B$. We write, for each $x \in K_{n+1}$, $\partial^- x$ and $\partial^+x$ for the negative and positive parts of $\partial x$ in the free preordered abelian group $K_n$. We call the basis of  $K$ \emph{normal} if for each basis element $x \in B_1$, the chains $\partial^-x$ and $\partial^+x$ are basis elements of $K_0$. Likewise, we call a morphism $f \colon K \longrightarrow L$ of free directed chain complexes \emph{normal} if the morphism $f_0 \colon K_0 \longrightarrow L_0$ preserves basis elements.

\subsection{Directed chain complexes and $\bm{\omega}$-categories}
We now describe an adjunction between the category $\mathbf{dCh}$ of directed chain complexes and the category $\bm{\omega}\mathbf{Cat}$ of (strict) $\omega$-categories and (strict) $\omega$-functors. 

The left adjoint functor $\lambda \colon \bm{\omega}\mathbf{Cat} \longrightarrow \mathbf{dCh}$ may be described as follows (cf.\ \cite[Definition 2.4]{steineromegacats}). On objects, $\lambda$ sends an $\omega$-category $C$ to the directed chain complex $\lambda C$ for which $(\lambda C)_n$ is the quotient of the free preordered abelian group on the set of $n$-cells of $C$  by the congruence generated by the relations $x + y = x \circ_k y$ for each $0 \leq k < n$ and each $k$-composable pair of $n$-cells $x,y$ of $C$; the boundary maps are given on a generating $(n+1)$-chain $x$ by $\partial x = t_nx - s_nx$. On morphisms, $\lambda$ sends an $\omega$-functor $f \colon C \longrightarrow D$ to the unique morphism of directed chain complexes $\lambda f \colon \lambda C \longrightarrow \lambda D$ sending each generating $n$-chain $x$ to the generating $n$-chain $f(x)$.

The right adjoint functor $\rho \colon \mathbf{dCh} \longrightarrow \bm{\omega}\mathbf{Cat}$ may be described as follows (cf.\ \cite[Definition 2.6]{steineromegacats}). On objects, $\rho$ sends a directed chain complex $K$ to the $\omega$-category $\rho K$ in which an $n$-cell $x$ is a table of chains in $K$
\begin{equation*}
x = 
\begin{pmatrix}
x_0^- & \cdots & x_n^- \\
x_0^+ & \cdots & x_n^+
\end{pmatrix}
\end{equation*}
such that, for $\alpha \in \{-,+\}$, each $x_k^\alpha \in K_k^*$, $\partial x_{k+1}^\alpha = x_k^+ - x_k^-$, and $x_n^- = x_n^+ =\vcentcolon x_n$; for each $0 \leq k < n$, the $k$-source (resp.\ $k$-target) of such an $n$-cell $x$ is the table $s_kx$ (resp.\ $t_kx$) whose first $k$ columns are the same as those of $x$ and whose last column is determined by $(s_kx)_k = x_k^-$ (resp.\ $(t_kx)_k = x_k^+$); for a pair of $n$-cells $x,y$ with $t_kx=s_ky$, their $k$-composite is the table
\begin{equation*}
x\circ_k y = 
\begin{pmatrix}
x_0^- & \cdots & x_k^- & x_{k+1}^- + y_{k+1}^- & \cdots & x_n^- + y_n^- \\
y_0^+ & \cdots & y_k^+ & x_{k+1}^+ + y_{k+1}^+ & \cdots & x_n^+ + y_n^+
\end{pmatrix},
\end{equation*}
and the identity $(n+1)$-cell of an $n$-cell $x$ is the table whose first $n+1$ columns are the same as those of $x$ and whose last column is a pair of zeros. 
On morphisms, $\rho$ sends a morphism of directed chain complexes $f \colon K \longrightarrow L$ to the $\omega$-functor $\rho f \colon \rho K\longrightarrow \rho L$ which sends a $n$-cell $x$ as above to the $n$-cell
\begin{equation*}
\rho f(x) = 
\begin{pmatrix}
f(x_0^-) & \cdots & f(x_n^-) \\
f(x_0^+) & \cdots & f(x_n^+)
\end{pmatrix}.
\end{equation*}

\subsection{Augmented directed complexes} \label{adcsec}
Observe that the functor $\lambda \colon \bm{\omega}\mathbf{Cat} \longrightarrow \mathbf{dCh}$ sends the terminal $\omega$-category $1$ to the directed chain complex $\mathbb{Z}[0]$, i.e.\ the abelian group of integers, with its natural order, concentrated in degree zero. Moreover, the $\omega$-category $\rho(\mathbb{Z}[0])$ is isomorphic to the $0$-category (i.e.\ set) of natural numbers (via the bijection $\begin{psmallmatrix} n\\n \end{psmallmatrix} \longleftrightarrow n$), and the unit map $\eta_1 \colon 1 \longrightarrow \rho\lambda 1 =  \rho (\mathbb{Z}[0]) \cong \mathbb{N}$ picks out the number $1$.

It follows that the adjunction $\lambda \dashv \rho$ yields a new adjunction
\begin{equation*}
\xymatrix{
\mathbf{ADC} \vcentcolon=\mathbf{dCh}/\mathbb{Z}[0] \ar@<-1.5ex>[rr]^-{\hdash}_-{\nu} && \ar@<-1.5ex>[ll]_-{\lambda} \bm{\omega}\mathbf{Cat}
}
\end{equation*}
which we shall describe presently (cf.\ \cite[Theorem 2.11]{steineromegacats}).

As indicated in the above diagram, we denote the slice category $\mathbf{dCh}/\mathbb{Z}[0]$ by $\mathbf{ADC}$; we call it the category of \emph{augmented directed complexes}. An augmented directed complex is therefore a directed chain complex $K$ equipped with a monotone augmentation of its underlying chain complex $\varepsilon\colon K_0 \longrightarrow \mathbb{Z}$, where $\mathbb{Z}$ has its natural order. A morphism of augmented directed complexes is an augmentation-preserving morphism of directed chain complexes. 

\begin{remark}\label{adcrmk}
The above definition of augmented directed complex is slightly stronger than Steiner's original definition \cite[Definition 2.2]{steineromegacats}, which amounts to a directed chain complex $K$ equipped with an augmentation of its underlying chain complex $\varepsilon \colon K_0 \longrightarrow \mathbb{Z}$ which is \emph{not} assumed to be monotone.
Steiner's category of augmented directed complexes is isomorphic to the slice category $\mathbf{dCh}/\mathbb{Z}^\sharp[0]$, where $\mathbb{Z}^\sharp$ denotes the abelian group of integers with the chaotic preorder (in which all elements are positive). 

While this slight difference makes no material impact on the main results of the present paper, we have chosen to adopt the definition which seems to us the more natural.  This stronger notion has been called a \emph{decent} augmented directed complex in \cite[\S2.17]{arajoint},\footnote{Ara and Maltsiniotis describe the notion of decent augmented directed complex as a ``notion peut-être plus pertinente que celle de complexe dirigé augmenté'' \cite[p.\ 7]{arajoint}.} and is called simply an augmented directed complex (without comment on the difference to Steiner's original definition) in \cite[Definition 3.3.12]{orientedcategorytheory}. We note that all augmented directed complexes that arise in either theory or practice have monotone augmentations.
\end{remark}

We continue to denote the new left adjoint by $\lambda$; it sends an $\omega$-category $C$ to the directed chain complex $\lambda C$ equipped with the augmentation $\lambda C \longrightarrow \lambda1 = \mathbb{Z}[0]$ given by the image under $\lambda$ of the unique $\omega$-functor $C \longrightarrow 1$. The new right adjoint $\nu$ (cf.\ \cite[Definition 2.8]{steineromegacats}) sends an augmented directed complex $K$ to the $\omega$-category $\nu K$ defined by the following pullback square of $\omega$-categories,
\begin{equation*}
\cd{
\nu K \ar[r] \ar[d] \fatpullbackcorner & \rho K \ar[d]^-{\rho(\varepsilon)}  \\
1 \ar[r]_-{\eta_1} & \rho (\mathbb{Z}[0])
}
\end{equation*}
that is, $\nu K$ is the sub-$\omega$-category of $\rho K$ consisting of those cells $x$ for which $\varepsilon(x_0^-) = \varepsilon(x_0^+) = 1$.

\begin{observation} \label{cocatobs}
The adjunction $\lambda \dashv \nu$ is entirely determined by a certain co-$\omega$-category internal to $\mathbf{ADC}$, being the image under $\lambda$ of the universal co-$\omega$-category in $\bm{\omega}\mathbf{Cat}$. Thus the $\omega$-category structure of $\nu K$ is determined as follows. The $n$-cells of $\nu K$ are in natural bijection with morphisms of augmented directed complexes $\lambda(\mathbb{D}^n) \longrightarrow K$, where $\mathbb{D}^n$ denotes the $n$-globe, $k$-sources and $k$-targets of $n$-cells (for $0 \leq k < n$) are given by precomposition with the image under $\lambda$ of the two inclusions $\mathbb{D}^k \longrightarrow \mathbb{D}^n$, $k$-composition of $n$-cells (for $0 \leq k < n$) is given by precomposition with the image under $\lambda$ of the $\omega$-functor $\mathbb{D}^n \longrightarrow \mathbb{D}^n \sqcup_{\mathbb{D}^k} \mathbb{D}^n$ that picks out the composite $n$-cell, and identity $(n+1)$-cells are given by precomposition with the image under $\lambda$ of the projection $\mathbb{D}^{n+1} \longrightarrow \mathbb{D}^n$. Cf.\ \cite[Examples 3.9 and 4.7]{steineromegacats}.
\end{observation}

A \emph{basis} of an augmented directed complex is simply a basis of its underlying directed chain complex (cf.\ \cite[Definition 3.1]{steineromegacats}). We call an augmented directed  complex \emph{free} if it admits a basis. We say that a basis of an augmented directed complex $K$ is \emph{normal} if it is normal as a basis of the underlying directed chain complex and if the augmentation $\varepsilon \colon K_0 \longrightarrow \mathbb{Z}$ sends each basis element of $K_0$ to $1$. 
The following proposition tells us that, when working with augmented directed complexes with normal bases, we may effectively ignore the augmentations.

\begin{proposition} \label{normalprop}
The category of augmented directed complexes with normal bases is isomorphic to the category of directed chain complexes with normal bases and normal morphisms between them.
\end{proposition}
\begin{proof}
By definition, if an augmented directed complex has a normal basis, so too has its underlying directed chain complex. 
Let $K$ and $L$ be augmented directed complexes with normal bases and let $f$ be a morphism between their underlying directed chain complexes. Since the basis elements of $K_0$ and $L_0$ are precisely the elements with augmentation $1$, the homomorphism $f_0 \colon K_0 \longrightarrow L_0$ preserves the augmentations of $K$ and $L$ if and only if it preserves basis elements. Hence the forgetful functor $\mathbf{ADC} \longrightarrow \mathbf{dCh}$ restricts to a fully faithful functor from the category of augmented directed complexes with normal bases to the category of directed chain complexes with normal bases and normal morphisms. 

This functor is bijective on objects because any directed chain complex $K$ with normal basis admits a unique augmentation $\varepsilon \colon K_0 \longrightarrow \mathbb{Z}$ sending each basis element to $1$; this indeed defines an augmentation since, for each basis element $x$ of $K_1$, we have
\begin{equation*}
\varepsilon(\partial x) = \varepsilon(\partial^+x - \partial^- x) = \varepsilon(\partial^+x) - \varepsilon(\partial^- x)= 1-1 = 0,
\end{equation*}
using that $\partial^-x$ and $\partial^+x$ are basis elements of $K_0$ by normality of the basis of $K$.
\end{proof}

\subsection{Unitality} \label{unitalsubsec}
Let $K$ be an augmented directed complex with basis $B$. Each element $x\in K_n$ $(n\geq 0)$ yields an $n$-cell $\left<x\right>$ in the $\omega$-category $\rho K$ defined as the following table:
\begin{equation*}
\left<x\right> =
\begin{pmatrix}
(\partial^-)^nx & \cdots & \partial^-x & x \\
(\partial^+)^nx & \cdots & \partial^+x & x 
\end{pmatrix}.
\end{equation*}
We say that the basis $B$ is \emph{unital} \cite[Definition 3.4]{steineromegacats} if, for each $n \geq 0$ and each basis element $x \in B_n$, the $n$-cell $\left<x\right>$ of $\rho K$ belongs to the sub-$\omega$-category $\nu K$, i.e.\ if $\varepsilon((\partial^-)^nx) = \varepsilon((\partial^+)^nx) = 1$. The cells $\left<x\right>$ for $x \in B$ are called the \emph{atoms} of $\nu K$ (cf.\ \cite[Definition 3.2]{steineromegacats}).

Note that the $n=0$ and $n=1$ cases of the definition of unital basis together amount  to the definition of a normal basis of an augmented directed complex. Hence any unital basis of an augmented directed complex is normal.

\subsection{Loop-freeness} \label{loopfreesubsec}
A basis $B$ of an augmented directed complex is said to be \emph{loop-free} \cite[Definition 3.5]{steineromegacats} if, for each $n \geq 0$, there exists a partial order $\triangleleft$ on $B$ such that, for $x,y \in B$, $x \mathrel{\triangleleft} y$ whenever $\left<x\right>_n^+ \wedge \left<y\right>_n^- \neq 0$. A \emph{Steiner complex} is an augmented directed complex with a loop-free unital basis. It is to the class of Steiner complexes that the fundamental theorems of Steiner theory apply; we shall refer to the following two.

\begin{theorem}[Steiner, {\cite[Theorem 5.6]{steineromegacats}}] \label{steiner1}
The restriction of the functor $\nu \colon \mathbf{ADC} \longrightarrow \bm{\omega}\mathbf{Cat}$ to the full subcategory of Steiner complexes is fully faithful.
\end{theorem}
\begin{theorem}[Steiner, {\cite[Theorem 6.1]{steineromegacats}}] \label{steiner2}
For every Steiner complex $K$, the \allowbreak$\omega$-category $\nu K$ is freely generated by its atoms.
\end{theorem}
For the notion of free generation of an $\omega$-category by a subset of its cells, see \cite[\S4]{streetaos}.

We now define a weaker loop-freeness condition, lifted from Axiom 3(a) in \cite[\S1]{paritycomplexes}.
\begin{definition} \label{triangledef}
A basis $B$ of an augmented directed complex is \emph{weakly loop-free} if, for each $n\geq 1$, there exists a partial order $\triangleleft$ on $B_n$ such that, for $x,y \in B_n$, $x \mathrel{\triangleleft} y$  whenever $\partial^+ x \wedge \partial^- y \neq 0$. 
\end{definition}
\begin{definition}
A \emph{weak Steiner complex} is an augmented directed complex with a weakly loop-free unital basis.
\end{definition}

The following proposition justifies this terminology.

\begin{proposition} \label{loopprop}
A loop-free basis of an augmented directed complex is weakly loop-free.
\end{proposition}
\begin{proof}
Let $B$ be a loop-free basis of an augmented directed complex, let $n\geq 1$, and let $\triangleleft$ be a partial order on $B$ such that $x \mathrel{\triangleleft} y$ whenever $\left<x\right>_{n-1}^+ \wedge \left<y\right>_{n-1}^- \neq 0$. Since for $x \in B_n$, $\left<x\right>_{n-1}^- = \partial^-x$ and $\left<x\right>_{n-1}^+ = \partial^+ x$, the restriction of the partial order $\triangleleft$ to $B_n$ witnesses that the basis $B$ is weakly loop-free.
\end{proof}

A basis $B$ of an augmented directed complex is said to be \emph{strongly loop-free} \cite[Definition 3.6]{steineromegacats} if there exists a partial order $\blacktriangleleft$ on $B$ such that, for $x,y \in B$, $x \mathrel{\blacktriangleleft} y$ whenever $x \leq \partial^- y$ or $y \leq \partial^+ x$. A \emph{strong Steiner complex} is an augmented directed complex with a strongly loop-free unital basis. By \cite[Proposition 3.7]{steineromegacats}, any strongly loop-free basis is loop-free, and hence also weakly loop-free by Proposition \ref{loopprop}.

The following lemma shows that a special case of Street's excision of extremals algorithm (cf.\ \cite[Theorem 4.1]{paritycomplexes}) can be carried out within augmented directed complexes with weakly loop-free bases. 
\begin{lemma} \label{excision}
Let $K$ be an augmented directed complex with a weakly loop-free basis. For each $n \geq 1$ and each $n$-cell $x$ in $\nu K$, there exists a decomposition $x = x^1 \circ_{n-1} \cdots \circ_{n-1} x^k$ in $\nu K$ such that, for each $1 \leq i \leq k$,  $x^i$ is an $n$-cell in $\nu K$ for which $(x^i)_n =\vcentcolon b_i$ is a basis element of $K_n$, such that $x_n = b_1 + \cdots + b_k$, and such that $\partial^+b_i\wedge \partial^-b_j = 0$ whenever $i \geq j$.
\end{lemma}
\begin{proof}
Write $x_n = b_1 +  \cdots + b_k $ as a sum of basis elements of $K_n$, which, by weak loop-freeness of the basis of $K$, we may suppose to be ordered in such a way that $\partial^+b_i \wedge \partial^-b_j = 0$ whenever $i \geq j$. The desired decomposition now follows from \cite[Proposition 5.1]{steineromegacats}.
\end{proof}

\subsection{Well-formedness}
Before we state the final propositions of this section, we first adapt some terminology from Street \cite[p.\ 318]{paritycomplexes} and Forest \cite[p.\ 45]{forest}.

\begin{definition} \label{radicaldef}
An element of a free commutative monoid is \emph{radical} if it is a sum of distinct basis elements.
\end{definition}

\begin{remark}
This terminology generalises the classical usage in the case of the free commutative monoid of strictly positive integers under multiplication, whose basis elements are the prime numbers.
\end{remark}

Note that the sum of a family of radical elements is radical if and only if the family is pairwise disjoint.

\begin{definition} \label{wellformeddef}
Let $K$ be an augmented directed complex with basis $B$. An element $x \in K_{n}$ is \emph{well-formed}\footnote{Forest called such elements ``fork-free'', but we have chosen to adhere to Street's original terminology.} if (i) $n=0$ and $\varepsilon(x) = 1$, or if (ii) $n > 0$, $x$ is radical in the free commutative monoid $K_n^*$, and for all basis elements $y,z\in B_{n}$ such that $y,z \leq x$, we have $\partial^-y \wedge \partial^-z = \partial^+y \wedge \partial^+z = 0$ whenever $y \neq z$.
\end{definition}

The following proposition and its corollaries underlie the comparison of augmented directed complexes and parity complexes. The proposition was proven by Forest  under the stronger hypothesis that $K$ is a Steiner complex \cite[Lemma 3.4.5]{forest}; cf.\ also \cite[Theorem 6.2(iii)]{steineromegacats}.

\begin{proposition} \label{triangleorderprop}
Let $K$ be an augmented directed complex with a weakly loop-free basis and let $x$ be an $n$-cell 
in $\nu K$. For each $0 \leq m \leq n$ and $\alpha \in \{-,+\}$, the element $x_m^\alpha$ is well-formed.
\end{proposition}
\begin{proof}
We prove by induction on $n \geq 0$ that, for every $n$-cell $x$ in $\nu K$, the element $x_n \vcentcolon= x_n^- = x_n^+$ in $K_n$ is well-formed. Since $x_m^- = (s_mx)_m$ and $x_m^+ = (t_mx)_m$, and since $s_mx$ and $t_mx$ are $m$-cells in $\nu K$, this implies the theorem.

The $n=0$ case is immediate from the definitions. Now let $n > 0$. If $x_n$ is a basis element or zero, then $x_n$ is well-formed and we are done. Else, by Lemma \ref{excision}, there exists a decomposition $x = x^1 \circ_{n-1} \cdots \circ_{n-1} x^k$ in $\nu K$ where each $x^i$ is an $n$-cell in $\nu K$ with $(x^i)_n = b_i$ a basis element of $K_n$, ordered in such a way that $\partial^+b_i \wedge \partial^- b_j = 0$ whenever $i \geq j$. 
The proof now proceeds as in the proof of \cite[Lemma 3.4.5]{forest}.
\end{proof}
\begin{corollary} \label{triangleordercor}
Let $K$ be a weak Steiner complex. For each $n \geq 0$, each basis element $x$ in $K_{n}$, and each $0 \leq m \leq n$, the elements $(\partial^-)^mx$ and $(\partial^+)^mx$ are well-formed.
\end{corollary}
\begin{proof}
As $K$ has unital basis, the $n$-cell $\left<x\right>$ belongs to $\nu K$. Since $\left<x\right>_{n-m}^- = (\partial^-)^mx$ and $\left<x\right>_{n-m}^+ = (\partial^+)^mx$, the result follows from Poposition \ref{triangleorderprop}.
\end{proof}
\begin{corollary} \label{triangleordercor2}
Let $K$ be an augmented directed complex with a weakly loop-free basis and let $x^1,\ldots,x^k$ be an $m$-composable $k$-tuple of $n$-cells in $\nu K$ for some $n \geq 1$ and $0 \leq m < n$. Then, for each $m < p \leq n$ and $\alpha \in \{-,+\}$, the elements $(x^1)^\alpha_p,\ldots,(x^k)^\alpha_p$ are pairwise disjoint.
\end{corollary}
\begin{proof}
Since $x^1,\ldots,x^k$, and $x^1 \circ_m \cdots \circ_m x^k$ are $n$-cells in $\nu K$, Proposition \ref{triangleorderprop} tells us that, for each $m < p \leq n$ and $\alpha \in \{-,+\}$, the elements $(x^1)^\alpha_p,\ldots,(x^k)^\alpha_p$, and $(x^1 \circ_m \cdots \circ_m x^k)^\alpha_p = (x^1)^\alpha_p + \cdots + (x^k)^\alpha_p$ are well-formed, and hence in particular radical. But the sum of a family of radical elements is radical if and if the family is pairwise disjoint. Hence the elements $(x^1)^\alpha_p,\ldots,(x^k)^\alpha_p$ are pairwise disjoint.
\end{proof}

The final proposition of this section was also proven by Forest under the stronger hypothesis that $K$ is a Steiner complex \cite[Lemma 3.4.11]{forest}. We shall use it in Remark \ref{verityrmk} below to show that our definition of morphism of (weak) parity complexes agrees with Verity's definition.

\begin{proposition} \label{triangleorderprop2}
Let $K$ be an augmented directed complex with a weakly loop-free basis and let $x$ be an $n$-cell in $\nu K$. For each $0 \leq m < n$, $\alpha \in \{-,+\}$, and each basis element $b \leq x_{m+1}^\alpha$, we have $x_m^- \wedge \partial^+b = 0$ and $x_m^+ \wedge \partial^-b = 0$.
\end{proposition}
\begin{proof}
Forest's proof of this conclusion in \cite[Lemma 3.4.11]{forest} carries through under the weaker hypothesis of weak loop-freeness, since it depends only on \cite[Lemma 3.4.5]{forest}, which we have proved to hold under this weaker hypothesis in Proposition \ref{triangleorderprop}, and on the result stated in Lemma \ref{excision}.
\end{proof}

\section{Additive parity complexes} \label{secapc}
We begin our study of parity complexes by analysing the structure of free augmented directed complexes, which leads us to the purely combinatorial notion of additive parity complex. The basic strategy of our analysis is to eliminate all reference to negative elements, phrasing everything in terms of positive elements. We realise these positive elements as multisets, which allows us to imitate the combinatorial language of Street's theory of parity complexes.

\subsection{Multisets}
A free commutative monoid with basis $B$ is canonically isomorphic to the monoid of functions $B \longrightarrow \mathbb{N}$ of finite support under pointwise addition. Such functions admit a combinatorial interpretation as the characteristic functions of finite multisets of $B$. A \emph{multiset} of a set $B$ is a collection of elements of $B$ in which elements may occur with finite multiplicity; in short, a multiset is a ``subset with repetition''. Note that subsets of $B$ are precisely the multisets of $B$ in which each element occurs with multiplicity at most one. 

We henceforth take the set $\mathcal{M}_f(B)$ of finite multisets of a set $B$ as our preferred model for the free commutative monoid on $B$, and hence also for the positive cone of the free preordered abelian group on $B$. The monoid operation on $\mathcal{M}_f(B)$ is disjoint union $\sqcup$ of multisets and the zero element is the empty subset $\emptyset$; as a free commutative monoid, its basis elements are the singleton subsets of $B$ and its radical elements are the subsets of $B$. We make frequent use of the ``multiset difference'' operation $S \setminus T$, defined to be the positive part of the element $S - T$, i.e.\ $S \setminus T \vcentcolon= (S-T)\vee 0$; this generalises the familiar set difference operation on subsets. The crucial advantage of the algebra of multisets over the algebra of subsets is that the operation of disjoint union of multisets is defined universally, not only for disjoint pairs of subsets.

\subsection{Additive parity structures}
Let $K$ be a free directed chain complex with basis $B$. (We shall consider augmentations at the end of this section.) Observe that, for each $n \geq 0$, the preordered abelian group $K_n$ can be recovered up to isomorphism from its basis $B_n$, and  the boundary homomorphism $\partial \colon K_{n+1} \longrightarrow K_n$ can be recovered from its values $\partial x \in K_n$ on basis elements $x \in B_{n+1}$, which values can be recovered from their disjoint negative and positive parts $\partial^-x$ and $\partial^+x$ in $K_n^* \cong \mathcal{M}_f(B_n)$ as $\partial x = \partial^+x - \partial^-x$. This motivates the following preliminary definition.

\begin{definition} An \emph{additive parity structure} is a graded set $B = (B_n)_{n\geq 0}$ together with, for each element $x \in B_{n+1}$ ($n \geq 0$), a disjoint pair of finite multisets $\partial^-x$ and $\partial^+x$ of $B_{n}$. \end{definition}

We denote the unique homomorphic extensions of the functions $\partial^-,\partial^+ \colon B_{n+1} \longrightarrow \mathcal{M}_f(B_n)$ by $\Phi^-,\Phi^+ \colon \mathcal{M}_f(B_{n+1}) \longrightarrow \mathcal{M}_f(B_n)$ respectively. For each finite multiset $S \in \mathcal{M}_f(B_{n+1})$, we define 
\begin{equation*}
\partial^-(S) \vcentcolon = \Phi^-(S)\setminus \Phi^+(S), \qquad \partial^+(S) \vcentcolon = \Phi^+(S)\setminus \Phi^-(S).
\end{equation*}
 Note that, for each element $x \in B_{n+1}$,  we have $\partial^-\{x\} = \partial^-x$ and $\partial^+\{x\} = \partial^+x$  thanks to the disjointness hypothesis $\partial^-x \cap \partial^+x = \emptyset$ in the definition of  additive parity structure.

\subsection{Additive parity complexes}
In order for an additive parity structure to generate a chain complex, it must satisfy some equivalent of the identity $\partial\partial = 0$. Writing $\partial = \partial^+ - \partial^-$, we see that this identity holds precisely when
\begin{equation*}
0 = \partial\partial  = \partial(\partial^+  - \partial^- ) = \partial\partial^+- \partial\partial^-, 
\end{equation*}
that is, precisely when
\begin{equation*}
\partial\partial^- = \partial\partial^+. 
\end{equation*}
Equating negative and positive parts, we see that the above equation holds if and only if the following pair of equations holds:
\begin{equation*}
\partial^-\partial^- = \partial^-\partial^+, \qquad \partial^+\partial^- = \partial^+\partial^+.
\end{equation*}
This motivates the following definition.

\begin{definition} \label{apcdef}
An \emph{additive parity complex} is an additive parity structure $B$ such that, for each $x \in B_{n+2}$ ($n \geq 0$), the following pair of equations of multisets of $B_{n}$ holds:
\begin{equation*}
\partial^-\partial^-x = \partial^-\partial^+x, \qquad \partial^+\partial^-x = \partial^+\partial^+x.
\end{equation*}
\end{definition}

We call this pair of equations the \emph{globularity axiom} for additive parity complexes. 
We write $\mathbb{Z}B$ for the directed chain complex freely generated by an additive parity complex $B$. For each $n \geq 0$, its $n$\textsuperscript{th} chain group $(\mathbb{Z}B)_n = \mathbb{Z}B_n$ is the free preordered abelian group on the set $B_n$, and the boundary map $\partial \colon \mathbb{Z}B_{n+1} \longrightarrow \mathbb{Z}B_n$ is the unique group homomorphism given on basis elements by $\partial x = \partial^+x - \partial^-x$. 

To prove that this construction underlies an equivalence of categories, we must first define morphisms of additive parity complexes.

\subsection{Movement}
A morphism of free directed chain complexes $f \colon K \longrightarrow L$ is a family of monotone group homomorphisms $f_n \colon K_n \longrightarrow L_n$ for $n \geq 0$, such that the identity $\partial f_{n+1} = f_n \partial$ holds for all $n\geq 0$. Each homomorphism $f_n$ can be recovered from its value on basis elements. 
Once again writing $\partial = \partial^+ - \partial^-$, we see that the above identity holds precisely when
\begin{equation*}
\partial f_{n+1} = f_n\partial^+ - f_n\partial^-.
\end{equation*}
Equating negative and positive parts yields the pair of equations
\begin{equation*}
\partial^-f_{n+1} = f_n\partial^- \setminus f_n\partial^+, \qquad \partial^+f_{n+1} = f_n\partial^+ \setminus f_n\partial^-.
\end{equation*}
This motivates the following pair of definitions.
\begin{definition} \label{movesdef}
Let $B$ be an additive parity complex, let $S$ be a finite multiset of  $B_{n+1}$, and let $M$ and $P$ be finite multisets of $B_n$. We say that $S$ \emph{moves} $M$ to $P$ if the following pair of equations holds:
\begin{equation*}
\partial^-(S) = M\setminus P, \qquad \partial^+(S) = P \setminus M.
\end{equation*}
\end{definition}

Observe that $S$ moves $M$ to $P$ in an additive parity complex $B$ if and only if the equation $\partial(S) = P-M$ holds in the free directed chain complex $\mathbb{Z}B$.

\subsection{The category of additive parity complexes}

\begin{definition}
A \emph{morphism of additive parity complexes} $f \colon B \longrightarrow C$ is a family of functions $f_n \colon B_n \longrightarrow \mathcal{M}_f(C_n)$ for $n \geq 0$, such that for each $x \in B_{n+1}$, the multiset $f_{n+1}(x)$ moves $f_n(\partial^-x)$ to $f_n(\partial^+x)$. 
\end{definition}
Note that, in this definition, we have denoted the unique homomorphic extension of the function $f_n \colon B_n \longrightarrow \mathcal{M}_f(C_n)$ by the same function symbol $f_n \colon \mathcal{M}_f(B_n) \longrightarrow \mathcal{M}_f(C_n)$.

We define  composition of morphisms of additive parity complexes so as to correspond to the composition of morphisms of free directed chain complexes.
Thus the composite of a pair of morphisms of additive parity complexes $f \colon B \longrightarrow C$ and $g \colon C \longrightarrow D$ is the morphism $g\circ f \colon B \longrightarrow D$ whose $n$\textsuperscript{th} component is the composite function
\begin{equation*}
\cd{
B_n \ar[r]^-{f_n} & \mathcal{M}_f(C_n) \ar[r]^-{g_n} & \mathcal{M}_f(D_n),
}
\end{equation*}
where the second factor is the unique homomorphic extension of the function $g_n \colon C_n \longrightarrow \mathcal{M}_f(D_n)$. 
The identity morphism for an additive parity complex $B$ is the morphism whose $n$\textsuperscript{th} component is the function $B_n \longrightarrow \mathcal{M}_f(B_n)$ which sends an element $x$ to the singleton subset $\{x\}$.

This defines the category of additive parity complexes, which we denote by $\mathbf{APC}$. The following proposition has been proved in the course of our exposition in this section.

\begin{proposition} \label{equivprop}
The category of additive parity complexes is equivalent to the category of free directed chain complexes. 
\end{proposition}

\subsection{Additive parity complexes and augmented directed complexes}
Let us now address augmentations. Let $1[0]$ denote the additive parity complex given by the basis of the directed chain complex $\mathbb{Z}[0]$; thus it consists only of a single element in dimension zero. We therefore define an \emph{augmentation} of an additive parity complex $B$ to be a morphism of additive parity complexes $B \longrightarrow 1[0]$. The following corollary is a consequence of Proposition \ref{equivprop} and of our definition of the category of augmented directed complexes as a slice category in \S\ref{adcsec}.

\begin{corollary}
The category of free augmented directed complexes is equivalent to the slice category $\mathbf{APC}/1[0]$.
\end{corollary}

Rather than working with augmentations, we may instead, in the spirit of Proposition \ref{normalprop}, impose conditions of normality on additive parity complexes and their morphisms. The following are direct translations of the notions of normality defined in \S\ref{dccs}.

We say that an additive parity complex $B$ is \emph{normal} if, for each $x\in B_1$, the multisets $\partial^-x$ and $\partial^+x$ are singleton subsets of $B_0$. We say that a morphism of additive parity complexes $f \colon B \longrightarrow C$ is \emph{normal} if the function $f_0 \colon B_0 \longrightarrow \mathcal{M}_f(C_0)$ sends each element to a singleton. We may now state the following corollary of Propositions \ref{normalprop} and \ref{equivprop}. 

\begin{corollary} \label{lemmagreen}
The category of augmented directed complexes with normal bases is equivalent to the category of normal additive parity complexes and normal morphisms between them.
\end{corollary}

\section{Weak parity complexes} \label{secwpc}

In the previous section, we realised the basis of any augmented directed complex $K$ as an additive parity complex, a structure which involves the assignment to each basis element $x$ of $K_{n+1}$ of two multisets $\partial^-x$ and $\partial^+x$ of $K_n$.  When $K$ is a weak Steiner complex, Corollary \ref{triangleordercor} implies that these multisets   are in fact \emph{subsets} of $K_n$. The purpose of this section is to describe combinatorially the category of weak Steiner complexes without recourse to multisets but purely within Street's theory of parity complexes.

\subsection{Parity structures} We begin with a preliminary definition.
\begin{definition}
A \emph{parity structure} is a graded set $C = (C_n)_{n \geq 0}$ together with, for each element $x \in C_{n+1}$ ($n \geq 0)$, a disjoint pair of finite subsets $x^-$ and $x^+$ of $C_n$.
\end{definition}

Just as every subset is a multiset, so too is every parity structure an additive parity structure, with $\partial^-x = x^-$ and $\partial^+x= x^+$. This grants us access to the theory of additive parity structures developed in the previous section, of which we make free use. 

\begin{example}
As mentioned above, the basis of any weak Steiner complex is a parity structure with $x^- = \partial^-x$ and $x^+ = \partial^+x$ by Corollary \ref{triangleordercor}.
\end{example}

We now introduce some notation. For $C$  a parity structure, $n >0$, and $S \subseteq C_{n}$, we define the operations
\begin{equation*}
S^- = \bigcup_{x \in S} x^-, \qquad S^+ = \bigcup_{x \in S} x^+,
\end{equation*}
and the derived operations
\begin{equation*}
S^\mp = S^-\setminus S^+, \qquad S^\pm = S^+\setminus S^-.
\end{equation*}
Note that, for each element $x$, we have $\{x\}^\mp = x^-$ and $\{x\}^\pm = x^+$ by the disjointness hypothesis $x^- \cap x^+ = \emptyset$ in the definition of parity structure.

\subsection{Well-formed subsets}
Translation between the theories of parity structures and additive parity structures is facilitated by the following notion (cf.\ Definition \ref{wellformeddef}).

\begin{definition} \label{wfdef2}
Let $C$ be a parity structure. A subset $S \subseteq C_n$ is \emph{well-formed} if (i) $n = 0$ and $S$ is a singleton, or (ii) $n > 0$ and for $x,y \in S$, we have $x^- \cap y^- = x^+ \cap y^+ = \emptyset$ whenever $x \neq y$.
\end{definition}

\begin{observation} \label{earlyobs}
If $K$ is a weak Steiner complex with underlying parity structure $C$, a finite subset $S \subseteq C_n$ ($n \geq 0$) is well-formed in this sense if and only if the corresponding element $\bigsqcup_{x \in S} x$ of $K_n^*$ is well-formed in the sense of Definition \ref{wellformeddef}.
\end{observation}
The following lemma indicates the importance of this notion.

\begin{lemma} \label{wflemma}
Let $C$ be a parity structure, $n >0$, and $S\subseteq C_n$ a finite subset. Then the following are equivalent:
\begin{enumerate}[\normalfont(i)]
\item $S$ is well-formed;
\item the multisets $\Phi^-(S)$ and $\Phi^+(S)$ are subsets;
\item $\Phi^-(S) = S^-$ and $\Phi^+(S) = S^+$.
\end{enumerate}
Moreover, if $S$ is well-formed, then $\partial^-(S) = S^\mp$ and $\partial^+(S) = S^\pm$.
\end{lemma}
\begin{proof}
Since $S$ is a subset, the multisets  $\Phi^-(S)$ and $\Phi^+(S)$ are by definition equal to the disjoint unions 
$$\Phi^-(S) = \bigsqcup_{x \in S} x^-, \qquad \Phi^+(S) = \bigsqcup_{x \in S} x^+.$$
A family of subsets is pairwise disjoint if and only if its disjoint union is a subset if and only if its disjoint union is equal to its union. Since for each $x \in S$,  $x^-$ and  $x^+$ are subsets, and since well-formedness of $S$ amounts to the pairwise disjointness of the two families above, this proves the equivalence (i)$\iff$(ii)$\iff$(iii).

Hence, if $S$ is well-formed, we have $$\partial^-(S) = \Phi^-(S)\setminus \Phi^+(S) = S^-\setminus S^+ = S^\mp,$$
and likewise $\partial^+(S) = S^\pm$.
\end{proof}

In order to achieve compatibility with the theory of additive parity complexes, we shall therefore endeavour to restrict our application of the operations sending a subset $S$ to the subsets $S^\mp$ and $S^\pm$ to well-formed subsets $S$.

\subsection{Globularity}
We now state an analogue for parity structures of the globularity axiom for additive parity complexes.

\begin{definition}
A parity structure $C$ is \emph{globular} if, for each $n \geq 2$ and $x \in C_{n}$, the following pair of equations holds:
\begin{equation*}
x^{-\mp} = x^{+\mp}, \qquad x^{-\pm} = x^{+\pm}.
\end{equation*}
\end{definition}

Since the above definition involves the application of the operations $S \longmapsto S^\mp, S^\pm$ to the subsets $x^-$ and $x^+$, we ought to require that these subsets be well-formed. Doing so has the following pleasant consequence. 

\begin{proposition} \label{globprop}
Let $C$ be a parity structure such that, for each $n \geq 1$ and $x \in C_n$, the subsets $x^-$ and $x^+$ are well-formed. Then $C$ is globular if and only if it is a normal additive parity complex.
\end{proposition}
\begin{proof}
Let $n \geq 2$ and $x\in C_n$. By Lemma \ref{wflemma}, well-formedness of $x^- = \partial^-x$ and $x^+=\partial^+x$ implies the identities $x^{-\mp} = \partial^-\partial^-x$, $x^{+\mp} = \partial^-\partial^+x$, $x^{-\pm} = \partial^+\partial^-x$ and $x^{+\pm} = \partial^+\partial^+x$. Hence $C$ is globular if and only if the following pair of equations holds
\begin{equation*}
\partial^-\partial^-x = \partial^-\partial^-x, \qquad \partial^+\partial^-x = \partial^+\partial^+x
\end{equation*}
that is, if and only if $C$ is an additive parity complex. Well-formedness of $x^-$ and $x^+$ for $x \in C_1$ amounts to normality of this additive parity complex.
\end{proof}

\begin{observation} \label{following}
Hence any parity structure $C$ which is globular and satisfies the hypothesis of Proposition \ref{globprop} freely generates an augmented directed complex $\mathbb{Z}C$ with a normal basis by Corollary \ref{lemmagreen}. Conversely, by Corollary \ref{triangleordercor}, the basis of any weak Steiner complex is a globular parity structure satisfying the hypothesis of Proposition \ref{globprop}. 
\end{observation}

\subsection{Unitality}
In order to state the analogue for parity structures of the unitality property of bases of augmented directed complexes, we recall the following notation from \cite[\S 4]{paritycomplexes}. 

Let $C$ be a parity structure, $n\geq 0$, and $x\in C_n$.  We define by downwards recursion on $0 \leq k \leq n$ a pair of families of subsets $\mu(x)_k$ and $\pi(x)_k$ as follows:
\begin{equation*}
\mu(x)_n = \pi(x)_n = \{x\},
\end{equation*}
\begin{equation*}
\mu(x)_{k-1} = \mu(x)_k^\mp, \qquad \pi(x)_{k-1} = \pi(x)_k^\pm.
\end{equation*}

\begin{definition}
A parity structure $C$ is \emph{unital} if for all $n\geq 0$, $x\in C_n$, and $0 \leq k \leq n$, the subsets $\mu(x)_k$ and $\pi(x)_k$ are well-formed.
\end{definition}
Note that, since $\mu(x)_{n-1} = x^-$ and $\pi(x)_{n-1} = x^+$ for $x\in C_n$, any unital parity structure satisfies the hypothesis of Proposition \ref{globprop}. In fact, more is true.

\begin{proposition} \label{proptop}
Any globular and unital parity structure freely generates an augmented directed complex with a unital basis.
\end{proposition}
\begin{proof}
Let $C$ be a globular and unital parity structure.
We noted in Observation \ref{following} that $C$ freely generates an augmented directed complex $\mathbb{Z}C$ with a normal basis. It remains to show that the basis of $\mathbb{Z}C$ is moreover unital.

Let $n\geq 0$ and $x \in C_n$. We see, by well-formedness of the sets $\mu(x)_k$ and $\pi(x)_k$, by Lemma \ref{wflemma}, and by induction on $0 \leq m \leq n$  that $(\partial^-)^{m}x = \mu(x)_{n-m}$ and $(\partial^+)^{m}x = \pi(x)_{n-m}$ for all $0 \leq m \leq n$. Hence $(\partial^-)^nx$ and $(\partial^+)^nx$ are well-formed subsets of $C_0$, which is to say, singletons, and therefore have augmentation $1$, since $\mathbb{Z}C$ has a normal basis.
\end{proof}

This proposition admits a partial converse.

\begin{proposition} \label{proptop2}
The basis of a weak Steiner complex is a globular and unital parity structure.
\end{proposition}
\begin{proof}
We noted in Observation \ref{following} that the basis $B$ of a weak Steiner complex is a globular parity structure. It remains to show that $B$ is unital.

Let $n\geq 0$ and $x \in B_n$. Corollary \ref{triangleordercor} and Observation \ref{earlyobs} imply that, for each $0 \leq m \leq n$, the multisets $(\partial^-)^mx$ and $(\partial^+)^mx$ are well-formed subsets. So by downwards induction on $0 \leq k \leq n$ and by Lemma \ref{wflemma}, we have that $\mu(x)_k = (\partial^-)^{n-k}x$ and $\pi(x)_k = (\partial^+)^{n-k}x$ and hence that these subsets are well-formed. 
\end{proof}

\subsection{Weak loop-freeness} 
It is a straightforward matter to translate the notion of weak loop-freeness for bases of augmented directed complexes (Definition \ref{triangledef}) into the language of parity structures (cf.\ Axiom 3(a) of \cite[\S 1]{paritycomplexes}).

\begin{definition}
A parity structure $C$ is \emph{weakly loop-free} if, for each $n \geq 1$, there exists a partial order $\triangleleft$ on $C_n$ such that, for $x,y \in C_n$, $x \triangleleft y$ whenever $x^+ \cap y^- \neq \emptyset$.
\end{definition}

It is immediate from the definitions that a globular and weakly loop-free parity structure satisfying the hypothesis of Proposition \ref{globprop} freely generates an augmented directed complex with a weakly loop-free basis, and conversely that the basis of a weak Steiner complex is a weakly loop-free parity structure.

\subsection{Weak parity complexes} Combining the above three properties, we arrive at the following definition. 
\begin{definition}
A \emph{weak parity complex} is a parity structure which is globular, unital, and weakly loop-free.
\end{definition}
The following proposition is a summary of the propositions stated and observations made so far in this section.
\begin{proposition} \label{bigprop}
Any weak parity complex freely generates a weak Steiner complex. Conversely, the basis of any weak Steiner complex is a weak parity complex.
\end{proposition}

To show that these constructions form an equivalence of categories, we must first define the category of weak parity complexes.

\subsection{Movement}
We first characterise movement, defined for additive parity complexes in Definition \ref{movesdef}, inside weak parity complexes. 
\begin{proposition} \label{movingprop}
Let $C$ be a weak parity complex, and let $S \subseteq C_{n+1}$ and $M,P \subseteq C_n$ be finite subsets. If $S$ is well-formed, then $S$ moves $M$ to $P$ if and only if the following pair of equations holds:
\begin{equation*}
S^\mp = M \setminus P, \qquad S^\pm = P \setminus M.
\end{equation*}
\end{proposition}
\begin{proof}
Since $S$ is well-formed, Lemma \ref{wflemma} implies that $S^\mp = \partial^-(S)$ and $S^\pm = \partial^+(S)$. %The result now follows from the definition of movement.
\end{proof}

\subsection{Morphisms of weak parity complexes}
We now extract the notion of morphism of parity complexes from the notion of normal morphism of additive parity complexes.

\begin{proposition} \label{bigmapprop}
Let $C$ and $D$ be weak parity complexes. A family of functions $f_n \colon C_n \longrightarrow \mathcal{M}_f(D_n)$ for $n \geq 0$ is a normal morphism of additive parity complexes if and only if the following two properties hold for all $n \geq 0$:
\begin{enumerate}[\normalfont(i)]
\item for each $x \in C_n$,  $f_n(x)$ is a well-formed subset of $D_n$,
\item for each $x \in C_{n+1}$, $f_{n+1}(x)$ moves $\bigcup_{y \in x^-} f_n(y)$ to $\bigcup_{y \in x^+}f_n(y)$.
\end{enumerate}
\end{proposition}
\begin{proof}
First suppose that the family $f = (f_n)_{n \geq 0}$ is a normal morphism of additive parity complexes. For each $x \in C_n$, the corresponding $\omega$-functor $\nu f \colon \nu (\mathbb{Z}C)\longrightarrow \nu (\mathbb{Z}D)$ sends the atom $\left<x\right>$, which is an $n$-cell in $\nu (\mathbb{Z}C)$ by unitality of $C$, to an  $n$-cell $f(\left<x\right>)$ in $\nu (\mathbb{Z}D)$ with $(f(\left<x\right>))_n = f_n(x)$. Proposition \ref{triangleorderprop} then implies that $f_n(x)$ is a well-formed subset of $D_n$, proving (i). We also have, for each $x \in C_{n+1}$, that $f_{n+1}(x)$ moves $f_n(x^-) = \bigsqcup_{y \in x^-} f_n(y)$ to $f_n(x^+) = \bigsqcup_{y \in x^+}f_n(y)$. But $f_n(x^-) = (s_nf(\left<x\right>))_n$ and $f_n(x^+) = (t_nf(\left<x\right>))_n$, which are subsets by Proposition \ref{triangleorderprop}, so these disjoint unions must in fact be unions. Hence $f_{n+1}(x)$ moves $\bigcup_{y \in x^-} f_n(y)$ to $\bigcup_{y \in x^+}f_n(y)$, proving (ii).

Now suppose that the family $f = (f_n)_{n \geq 0}$ satisfies the two properties (i) and (ii) for all $n \geq 0$. We shall prove by induction on $n \geq 0$ that the functions $f_k \colon C_k \longrightarrow \mathcal{M}_f(D_k)$ for $0 \leq k \leq n$ define a normal morphism of additive parity complexes $f_{\leq n} \colon C_{\leq n} \longrightarrow D_{\leq n}$ between the $n$-skeleta of $C$ and $D$ (defined by discarding all elements of dimension  $> n$), from which it follows that the entire family defines a normal morphism of additive parity complexes $f \colon C \longrightarrow D$. Note that any skeleton of a weak parity complex is a weak parity complex.

The $n=0$ case is immediate from property (i), which tells us that the function $f_0$ sends each element of $C_0$ to a singleton subset of $D_0$. Now suppose that $n > 0$ and that the functions $(f_k)_{0 \leq k\leq n}$ define a normal morphism of additive parity complexes $f_{\leq n} \colon C_{\leq n} \longrightarrow D_{\leq n}$. It remains to prove, for each $x \in C_{n+1}$, that $f_{n+1}(x)$ moves $f_n(x^-) = \bigsqcup_{y \in x^-} f_n(y)$ to $f_n(x^+) = \bigsqcup_{y \in x^+}f_n(y)$. By property (ii), it suffices to show that the unions $\bigcup_{y \in x^-} f_n(y)$ and $\bigcup_{y \in x^+}f_n(y)$ are disjoint unions, i.e.\ that both of the families $f_n(y)$ for $y \in x^-$ and $f_n(y)$ for $y \in x^+$ are pairwise disjoint. We prove this for the first family, the argument for the other family being dual.

Thanks to the weak loop-freeness of $C$, we may apply Lemma \ref{excision} to the $n$-cell $s_n(\left<x\right>)$ in $\nu(\mathbb{Z}C)$. This yields a decomposition $s_n(\left<x\right>) = Y^1 \circ_{n-1} \cdots \circ_{n-1} Y^k$ where each $Y^i$ is an $n$-cell in $\nu(\mathbb{Z}C)$ with $(Y^i)_n = y_i$, where $s_n(\left<x\right>)_n = x^- = \{y_1,\ldots,y_k\}$. By the induction hypothesis, the $\omega$-functor $\nu (\mathbb{Z}f_{\leq k}) \colon \nu(\mathbb{Z}C_{\leq n}) \longrightarrow \nu(\mathbb{Z}D_{\leq n })$ sends this decomposition to a composite  $f(Y^1) \circ_{n-1} \cdots \circ_{n-1} f(Y^k)$ of $n$-cells in $\nu(\mathbb{Z}D)$, such that, for each $1 \leq i\leq k$, we have that $f(Y^i)_n = f(y_i)$. Since $D$ is weakly loop-free, we may apply Corollary \ref{triangleordercor2} to deduce that the family of subsets $f(y_1),\ldots, f(y_k)$ is pairwise disjoint.
\end{proof}

We therefore make the following definition, in which we denote the set of finite subsets of a set $B$ by $\mathcal{P}_f(B)$, and we denote, as is standard, the unique union-preserving extension of a function $g \colon A \longrightarrow \mathcal{P}_f(B)$ by the same function symbol $g \colon \mathcal{P}_f(A) \longrightarrow \mathcal{P}_f(B)$.
\begin{definition} \label{mapdef}
A \emph{morphism of weak parity complexes} $f \colon C \longrightarrow D$ is a family of functions $f_n \colon C_n \longrightarrow \mathcal{P}_f(D_n)$ for $n \geq 0$, such that, for each $n\geq 0$ and $x \in C_n$, the subset $f_{n}(x)$ is well-formed and, if $n > 0$, $f_{n}(x)$ moves $f_{n-1}(x^-)$ to $f_{n-1}(x^+)$.
\end{definition}

\begin{remark} \label{verityrmk}
This definition is equivalent to Verity's notion of parity complex morphism found in \cite[Lemma 231]{complicial}. Indeed, the two definitions would be formally identical, but for the fact that the notion of movement used by Verity, being that of \cite[\S 2]{paritycomplexes}, is \emph{a priori} stronger than ours. According to this stronger notion, for a subset $S$ to move  $M$ to $P$ requires not only that $S^\mp = M \setminus P$ and $S^\pm = P \setminus M$ (as in Proposition \ref{movingprop}) but  also that $M \cap S^+ = \emptyset$ and $P \cap S^- = \emptyset$. 

To see that this stronger version of movement in fact holds in the movement clause of Definition \ref{mapdef}, we may argue as follows. Let $f \colon C \longrightarrow D$ be a morphism of weak parity complexes, and let $n > 0$ and $x \in C_n$. By Proposition \ref{bigmapprop}, $f$ is a morphism of additive parity complexes and hence the $\omega$-functor $\nu (\mathbb{Z}f)$ sends the atom $\left<x\right>$, which is an $n$-cell in $\nu (\mathbb{Z}C)$ since $C$ is unital, to the $n$-cell $f(\left<x\right>)$  in $\nu (\mathbb{Z}D)$. Since $D$ is weakly loop-free, we may apply Proposition \ref{triangleorderprop2} to the $n$-cell $f(\left<x\right>)$ to deduce that $f_{n-1}(x^-) \cap f_n(x)^+ = \emptyset$ and $f_{n-1}(x^+) \cap  f_n(x)^- = \emptyset$, so that $f_n(x)$ moves $f_{n-1}(x^-)$ to $f_{n-1}(x^+)$ in the stronger sense.
\end{remark}
We define the identity morphism for a weak parity complex $C$ to be the family of functions $C_n \longrightarrow \mathcal{P}_f(C_n)$ for $n \geq 0$ sending each element $x$ to the singleton $\{x\}$. It remains to define composition for morphisms of weak parity complexes. 
\begin{proposition} \label{compprop}
Let $f \colon C \longrightarrow D$ and $g \colon D \longrightarrow E$ be morphisms of weak parity complexes. Then the family of functions $C_n \longrightarrow \mathcal{P}_f(E_n)$ for $n\geq 0$ 
%\begin{equation*}
%\cd{
%C_n \ar[r]^-{f_n} & \mathcal{P}_f(D_n) \ar[r]^-{g_n} & \mathcal{P}_f(E_n)
%}
%\end{equation*}
sending each $x \in C_n$ to the union $\bigcup_{y \in f_n(x)} g_n(y)$ 
is a morphism of weak parity complexes $g\circ f \colon C \longrightarrow E$, which we call the \emph{composite} of $f$ and $g$. Moreover, this agrees with the composite of $f$ and $g$ as morphisms of additive parity complexes. 
\end{proposition}
\begin{proof}
By Proposition \ref{bigmapprop}, $f$ and $g$ are normal morphisms of additive parity complexes, and hence their composite $g \circ f$, defined as the family of functions $C_n \longrightarrow \mathcal{M}_f(E_n)$ for $n \geq 0$ sending $x \in C_n$ to the disjoint union $\bigsqcup_{y \in f_n(x)} g_n(y)$, is a normal morphism of additive parity complexes $C \longrightarrow E$. But Proposition \ref{bigmapprop} applied to this composite $g \circ f$ implies that the disjoint union  $\bigsqcup_{y \in f_n(x)} g_n(y)$ is a subset for each $x \in C_n$, whence the family of subsets $g_n(y)$ for $y \in f_n(x)$ is pairwise disjoint, whence their disjoint union is equal to their union. Therefore the composite $g\circ f$ is equal to the family of functions defined in the statement of the proposition, which by Proposition \ref{bigmapprop} and Definition \ref{mapdef} therefore defines a morphism of weak parity complexes. 
\end{proof}

We have now defined the category of weak parity complexes, which we denote by $\mathbf{WPC}$, and we may state the following theorem.
\begin{theorem} \label{athm}
The category of weak parity complexes is equivalent to the category of weak Steiner complexes. 
\end{theorem}
\begin{proof}
By construction (cf.\ Propositions  \ref{bigprop}, \ref{bigmapprop}, and \ref{compprop} and Definition \ref{mapdef}), the category of weak parity complexes is isomorphic to the full subcategory of the category of normal additive parity complexes and normal morphisms between them spanned by the bases of weak Steiner complexes. Hence the theorem follows from Corollary \ref{lemmagreen}.
\end{proof}

\subsection{Weak parity complexes and $\bm{\omega}$-categories}
Composing the equivalence of Theorem \ref{athm} with the restriction of the functor $\nu \colon \mathbf{ADC} \longrightarrow \bm{\omega}\mathbf{Cat}$ to the full subcategory of weak Steiner complexes, we obtain a functor $\mathcal{O} \colon \mathbf{WPC} \longrightarrow \bm{\omega}\mathbf{Cat}$ sending each weak parity complex $C$ to an $\omega$-category $\mathcal{O}C$. We conclude this section with a description of these $\omega$-categories within the theory of parity complexes (cf.\ \cite[\S 3]{paritycomplexes}).

As observed in \cite[Example 3.9]{steineromegacats} (see also \cite{steinersimple} and \cite[Chapitre 4]{arajoint}), the augmented directed complexes that make up the co-$\omega$-category representing the functor $\nu \colon \mathbf{ADC} \longrightarrow \bm{\omega}\mathbf{Cat}$ (see Observation \ref{cocatobs}) are strong Steiner complexes. Hence their bases are weak parity complexes, and it follows that the functor $\mathcal{O} \colon \mathbf{WPC} \longrightarrow \bm{\omega}\mathbf{Cat}$ is represented by a co-$\omega$-category in $\mathbf{WPC}$, which therefore allows a description of the $\omega$-categories $\mathcal{O}C$ as in Observation \ref{cocatobs}. We conclude this section by giving the results of this description.

Let $C$ be a weak parity complex. An $n$-cell $(M,P)$ in the $\omega$-category $\mathcal{O}C$ is a table of subsets
\begin{equation*}
(M,P) = 
\begin{pmatrix}
M_0 & \cdots & M_n \\
P_0 & \cdots & P_n
\end{pmatrix}
\end{equation*}
such that $M_k$ and $P_k$ and finite well-formed subsets of $C_k$, $M_{k+1}$ and $P_{k+1}$ move $M_k$ to $P_k$, and $M_n = P_n$. For each $0 \leq k < n$, the $k$-source (resp.\ $k$-target) of $(M,P)$ is the table whose first $k$ columns are the same as those of $(M,P)$ and whose last column is two copies of $M_n$ (resp.\ $P_n$), the $k$-composite of a $k$-composable pair of $n$-cells $(M,P)$ and $(N,Q)$ is the table
\begin{equation*}
(M,P) \circ_k (N,Q) = 
\begin{pmatrix}
M_0 & \cdots & M_k & M_{k+1}\cup N_{k+1} & \cdots & M_n \cup N_n \\
Q_0 & \cdots & Q_k & P_{k+1}\cup Q_{k+1} & \cdots & P_n \cup Q_n
\end{pmatrix}, 
\end{equation*}
and the identity $(n+1)$-cell on an $n$-cell $(M,P)$ is the table whose first $n+1$ columns are the same as those of $(M,P)$ and whose last column is two copies of the empty subset $\emptyset$. It follows from Corollary \ref{triangleordercor2} that the unions featuring in the definition of composition of cells are necessarily disjoint unions. We define the \emph{atoms} of $\mathcal{O}C$ to be the cells of the form
\begin{equation*}
\left<x\right> = (\mu(x),\pi(x)) = 
\begin{pmatrix}
\mu(x)_0 & \cdots & \mu(x)_n \\
\pi(x)_0 & \cdots & \pi(x)_n
\end{pmatrix}
\end{equation*}
for $x\in C_n$ and $n \geq 0$.

\section{Parity complexes} \label{secpc}
It now remains only to state our final definition and to prove our final theorem.

In accordance with \S\ref{loopfreesubsec}, we say that a parity structure $C$ is \emph{strongly loop-free} if there exists a partial order $\blacktriangleleft$ on $C$ such that, for $x,y \in C$, $x \blacktriangleleft y$ whenever $x \in y^-$ or $y \in x^+$. Note that any strongly loop-free parity structure is weakly loop-free. It is immediate from the definitions that a weak Steiner complex is strongly loop-free, i.e.\ is a strong Steiner complex, if and only if its underlying weak parity complex is strongly loop-free. This motivates the following definition.

\begin{definition}
A \emph{parity complex} is a parity structure which is globular, unital, and strongly loop-free.
\end{definition}

Note that any parity complex is a weak parity complex. We define the \emph{category of parity complexes} to be the full subcategory of $\mathbf{WPC}$ spanned by the parity complexes; we denote this category by $\mathbf{PC}$. The following theorem is now an immediate consequence of Theorem \ref{athm}.

\begin{theorem}
The category of parity complexes is equivalent to the category of strong Steiner complexes. 
\end{theorem}

Combining this theorem with Steiner's Theorem \ref{steiner2}, we conclude that parity complexes in our sense satisfy Street's main theorem on parity complexes \cite[Theorem 4.2]{paritycomplexes}.
\begin{corollary}
For every parity complex $C$, the $\omega$-category $\mathcal{O}C$ is freely generated by its atoms.
\end{corollary}

Likewise combining the above theorem with Steiner's Theorem \ref{steiner1}, we recover a version of Verity's \cite[Lemma 231]{complicial}.

\begin{corollary}
The functor $\mathcal{O} \colon \mathbf{PC} \longrightarrow \bm{\omega}\mathbf{Cat}$ is fully faithful.
\end{corollary}

\end{document}